\documentclass[letterpaper, 10 pt, conference]{ieeeconf}  % Comment this line out if you need a4paper

\IEEEoverridecommandlockouts                              % This command is only needed if
\usepackage{graphics} % for pdf, bitmapped graphics files
\usepackage{epsfig} % for postscript graphics files
\usepackage{mathptmx} % assumes new font selection scheme installed
\usepackage{times} % assumes new font selection scheme installed
\usepackage{amsmath} % assumes amsmath package installed
\usepackage{amssymb}  % assumes amsmath package installed
\usepackage{amsbsy,bm,fixmath}

\usepackage{cases}

\newtheorem{theorem}{Theorem}
\newtheorem{definition}[theorem]{Definition}
\newtheorem{lemma}[theorem]{Lemma}
\newtheorem{proposition}[theorem]{Proposition}

\newtheorem{example}{Example}
\newtheorem{remark}{Remark}

\title{\LARGE \bf Linear-Quadratic Mean Field Control: The Hamiltonian Matrix\\
 and Invariant Subspace Method    }

\author{Xiang Chen\qquad Minyi Huang% <-this % stops a space
\thanks{This work was  supported by NSERC.}% <-this % stops a space
\thanks{The authors are with the School of
Mathematics and Statistics, Carleton University, Ottawa, K1S 5B6 ON,
Canada (XiangChen@cmail.carleton.ca, mhuang@math.carleton.ca).
Proc. IEEE CDC'18 with corrected Figs. 1, 2.  }%
}

\begin{document}

\maketitle
\thispagestyle{empty}
\pagestyle{empty}

\begin{abstract}
This paper studies the existence and uniqueness of a solution to  linear quadratic  (LQ) mean field social optimization problems with uniform agents. We exploit a  Hamiltonian matrix structure of the associated ordinary differential equation (ODE) system
 and apply a subspace decomposition method to find the solution.
 This approach is effective for both the existence analysis and numerical computations.
 We further extend the decomposition method to  LQ mean field games.
\end{abstract}

\section{Introduction}
\label{chp:intro}

Mean field game (MFG) theory studies stochastic dynamic decision  problems involving a large number of noncooperative and individually insignificant agents, and provides a powerful methodology to reduce the complexity in designing  strategies \cite{HCM03}. For an overview of the theory and applications, the readers are referred to \cite{BFY13,CHM17,GLL11,HCM07,HMC06,LL07} and references therein.

There has existed a parallel development on mean field social optimization
where a large number of agents cooperatively minimize a social cost as the sum of individual costs.
Different from mean field games, the individual strategy selection of
an agent is not selfish and should take into account both self improvement and the aggregate impact on other agents' costs. Mean field social optimization problems have been studied in multi-agent
collective motion \cite{ACFK17,PRT15}, social consensus control \cite{NCMH13}, economic theory \cite{NM17}.
Other related literature includes  Markov decision processes using aggregate statistics and their mean field limit \cite{GGB12}, LQ mean field teams \cite{AM15}, LQ social optimization with
 a major player \cite{HN16}, mean field teams with
  Markov jumps \cite{WZ17},  social optimization with nonlinear diffusion dynamics \cite{SHM16},
 and  cooperative stochastic differential games \cite{YP06}.

   % and social optimization via norm minimization of an input-output mapping %\cite{VE17}.

%There are some similar social optimization problems presented in \cite{ACFK17},
%\cite{HN16} and \cite{SHM16}. In \cite{HN16}, the authors introduce a mean field LQG game model involving a %major player.  Another related work is \cite{ACFK17}, which considers nonlinear mean field optimal control to %minimize a social cost.

%\subsection{The  Social Optimization Problem}

In this paper, we consider social optimization in  an LQ model of uniform agents. The dynamics of agent $i$ are given by the stochastic differential equation (SDE):
\begin{align}
dx_i=(Ax_i+Bu_i)dt+DdW_i,    \label{doa}
        \quad t \geq 0, \quad 1\le i\le N.
\end{align}
 The state $x_i$ and the control $u_i$ are $n$ and $n_1$ dimensional vectors respectively. The initial states $\{x_i(0),\  1\leq i\leq N\}$ are independent. The noise processes  $\{W_i,\  1\leq i\leq N\}$ are $n_2$ dimensional independent standard Brownian motions, which are also independent of  $\{x_i(0),\  1\leq i\leq N\}$. The constant matrices $A$, $B$ and $D$  have compatible dimensions. Given a symmetric  matrix $M\ge 0$, the quadratic form $z^TMz$ is sometimes denoted by $|z|_M^2$. Denote $u:=(u_1,\cdots,u_N)$.

The individual cost for agent $i$ is given by
\begin{align}
J_i(u(\cdot))=E\int_0^\infty e^{-\rho t}[ |x_i-\Phi(x^{(N)})|_Q^2+u_i^TRu_i]dt, \label{ic}
\end{align}
where $\rho>0$,  $\Phi(x^{(N)})=\Gamma x^{(N)}+\eta$ and $x^{(N)}:=(1/N)
\sum_{i=1}^N x_i$ is  the mean field coupling term. The constant matrices or vectors $\Gamma$, $Q$, $R$ and $\eta$ have compatible dimensions, and $Q$, $R$ are symmetric.
The social cost is defined as
\begin{align}
J_{\rm soc}^{(N)}(u(\cdot))=\sum_{i=1}^N J_i(u(\cdot)). \nonumber % \label{sc1}
\end{align}
The minimization of the social cost is an optimal control problem. However, the exact solution requires centralized information for each agent. So
a solution of practical interest  is to find a set of  decentralized strategies
which has negligible optimality loss in minimizing
$J_{\rm soc}^{(N)}(u(\cdot))$ for large $N$ and the solution method has been developed in \cite{HCM12} under the following assumption:
(A1) $Q\ge 0$, $R>0$, $(A, B)$ is stabilizable and $(A,Q^\frac{1}{2})$ is detectable.

Under (A1), there  exists a unique solution   $\Pi\ge 0$ to the  algebraic Riccati equation (ARE):
\begin{align}
\rho \Pi=\Pi A+A^T\Pi-\Pi BR^{-1}B^T\Pi+Q. \label{ARE}
\end{align}

Denote $Q_\Gamma=\Gamma^TQ+Q\Gamma-\Gamma^TQ\Gamma$ and $\eta_\Gamma=(I-\Gamma^T)Q\eta$.
We introduce the Social Certainty Equivalence (SCE) equation system:
\begin{align}
&\frac{d\bar x}{dt}=(A-BR^{-1}B^T\Pi)\bar x-BR^{-1}B^Ts, \label{SCE1_1}\\
&\frac{ds}{dt}=Q_\Gamma\bar x+(\rho I-A^T+\Pi BR^{-1}B^T)s+\eta_\Gamma, \label{SCE2_1}
\end{align}
where $\bar x(0)=x_0$ is given and $s(0)=s_0$ is to be determined.  We look for $(\bar x, s)\in C_{\rho/2}([0,\infty),\mathbb R^{2n})$ (see Definition \ref{def:cr}).
If a finite time horizon $[0, T]$ is considered for \eqref{ic}, $s$ will have a terminal condition $s(T)$ and $\Pi$ will depend on time.
This results in a standard two point boundary value (TPBV) problem for
linear ordinary differential equations (ODEs).
Given the infinite horizon,  $s$ satisfies a growth condition instead of a terminal condition.

The key result in \cite{HCM12} under (A1) is that if \eqref{SCE1_1}-\eqref{SCE2_1} has a unique solution, the set of decentralized strategies
\begin{align}
\hat u_i=-R^{-1}B^T(\Pi x_i+s), \quad 1\le i\le N,  \label{oc}
\end{align}
 has asymptotic social optimality. In other words, centralized strategies can further reduce the cost $J_{\rm soc}^{(N)}(u(\cdot))$
by at most $o(N)$.
  In fact,
  \cite{HCM12} constructed a more general version of \eqref{SCE1_1}-\eqref{SCE2_1}    where
the parameter $A$ is randomized over the population and accordingly $\bar x$ in the equation of $s$ is replaced by a mean field averaging over
the nonuniform population.

%\begin{align}
%&\frac{d\bar x}{dt}=(A-BR^{-1}B^T\Pi)\bar x-BR^{-1}B^Ts. \label{SCE12_1}
%\end{align}

\subsection{Our Approach and Contributions}

After some transformation, the coefficient matrix of  \eqref{SCE1_1}-\eqref{SCE2_1} reduces to a Hamiltonian matrix which can be associated with an LQ optimal control problem with state weight matrix $-Q_\Gamma$. The connection to such an LQ control problem is remarkable since its state weight may not be positive semi-definite. When $Q_\Gamma\le 0$, existence and uniqueness of the solution has been proved \cite[Theorem 4.3]{HCM12} by a standard Riccati equation approach.
On the other hand, due to the intrinsic optimal control nature of the social optimization problem, one expects to obtain solvability of the SCE equation system under much more general conditions, which is the focus of this work.
Furthermore, our approach allows $Q$ to be indefinite.
LQ optimal control problems with indefinite state and/or control weight
is  a subject of considerable interest \cite{D15,RZ01,W71,YZ99}.

 We develop a new approach to
 analyze and compute  the solution of  (\ref{SCE1_1})-(\ref{SCE2_1}) for a general $Q_\Gamma$ by exploiting a Hamiltonian matrix structure and the well known
invariant subspace method \cite{BIM11}.
 Specifically, aided by the solution of a  continuous-time algebraic Riccati equation (ARE) with possibly indefinite state weight, we  decompose the Hamiltonian matrix into a block-wise triangular form where the stable eigenvalues are separated from the unstable ones.
To numerically solve the Riccati equation, we apply the
Schur method \cite{L79}.
  The approach of decomposing the  stable invariant subspace is  further extended  to solve  LQ mean field games; see \cite{B12,BS16,LZ08,MB17} for related literature on LQ mean field games.
 The main results of this paper have been reported  in \cite{C17,CH18}
 in an early form.

The organization of the paper is as follows. Section \ref{chp:socp} proves existence and uniqueness of a solution to the SCE equation system
and develops a computational method.  Section \ref{chp:mggp} extends the analysis to   LQ mean field games. Numerical examples are presented in Section \ref{sec:num}. Section \ref{chp:con} concludes the paper.

\section{Solution  of the Social Optimization Problem}
\label{chp:socp}

\subsection{Preliminaries on Algebraic Riccati Equations}

Let $S^n\subset \mathbb{R}^{n\times n}$ denote the set of symmetric matrices, and
$S^n_+\subset S^n$ the set of positive semi-definite matrices. Our later analysis   depends on an invariant subspace decomposition method which involves a class of continuous-time algebraic Riccati equations (ARE) of the form
\begin{align}
XA_o+A_o^TX-XMX  +Q_o=0,   \label{CARE}
\end{align}
where $A_o$, $Q_o$, $M$ are given matrices in $\mathbb R^{n\times n}$ with $Q_o\in S^n$ and $M\in S^n_+$. Note that $Q_o$ is not required to be positive semi-definite.  Denote the Hamiltonian matrix
\begin{align}
H_o=\begin{bmatrix}
A_o & -M\\
-Q_o & -A_o^T\\
\end{bmatrix}. \label{Ho}
\end{align}
Note that the eigenvalues of a  Hamiltonian matrix are distributed symmetrically about both the real axis  and the imaginary axis \cite{L84}.
If $H_o$ has no eigenvalue on the imaginary axis, the left and right open half planes each contain $n$ eigenvalues.

%We provide some standard definitions (see e.g. \cite{LR95}).

For a solution $X_+\in S^n$ of (\ref{CARE}), $X_+$ is the maximal real symmetric solution \cite{LR95} if for any  solution $X\in S^n$, $X_+-X\ge 0$.
A (real or complex) matrix is called stable if all its eigenvalues are in the open left half-plane; such an eigenvalue is also said to be stable.

%\begin{theorem}\cite[p. 196]{LR95}
%\label{MSS}
%Suppose  $M\geq 0$, $(A_o,M)$ is stabilizable and there exists a real symmetric %solution of (\ref{CARE}). Then for the maximal real symmetric solution $X_+$ %of (\ref{CARE}), $A_o-MX_+$ is stable if and only if $H_o$
%has no eigenvalues on the imaginary axis.
%\end{theorem}

\begin{proposition} \label{cor:ric}
If $(A_o,M)$ is stabilizable and $H_o$ has no eigenvalues on the imaginary axis (i.e., no eigenvalues with zero real parts), then there exists a unique maximal real symmetric solution $X_+$ and $A_o-MX_+$ is stable.
\end{proposition}

%{ \color{red} The proof of Corollary \ref{cor:ric} has been revised later, %but I feel it is unnatural. So I take notes here.}

\begin{proof}
By Theorem 9.3.1 in \cite[p. 239]{LR95}, there exists a unique almost stabilizing solution $X$ in $S^n$ (i.e., all eigenvalues of $A_o-MX$ are in the closed left half plane). Further applying Theorem 7.9.4 in \cite[p. 195-196]{LR95}, we obtain  a unique maximal real symmetric solution $X_+$ and $A_o-MX_+$ is stable.
\end{proof}

\subsection{The Transformation}
\label{sec:socd}

\begin{definition} \label{def:cr}
For integer  $k\ge 1$ and real number $r>0$, $C_{r}([0,\infty),\mathbb R^{k})$ consists of  all functions $f\in C([0,\infty),\mathbb R^{k})$ such that $\sup_{t\geq 0}|f(t)|e^{-r't}<\infty$, for some $0<r'< r$. Here $r'$ may depend on $f$.
\end{definition}

Denote
\begin{align}
\label{HA}
H_A=
\begin{bmatrix}
    A-\frac{\rho}{2}I & -BR^{-1}B^T \\
   -Q &  -A^T+\frac{\rho}{2}I  \\
\end{bmatrix}.\
\end{align}
We introduce the following standing assumption for the rest of this paper.

(SA)  $(A,B)$ is stabilizable, $R>0$, and $H_A$ has no eigenvalues with zero real parts.

Under (SA), we may solve a unique maximal solution $\Pi\in S^n$ from \eqref{ARE} such that $A-BR^{-1}B^T \Pi-\frac{\rho}{2} I$ is stable.
This ensures the construction of  \eqref{SCE1_1}-\eqref{SCE2_1}.
Note that we do not require $Q\ge 0$.

 Define
\begin{align*}
\tilde x=e^{-\rho t/2}\bar x ,\quad \tilde {s}=e^{-\rho t/2}s.
\end{align*}
 We obtain
\begin{align}
\begin{bmatrix}
   \frac{d\tilde{x}}{dt} \\
   \frac{d\tilde{s}}{dt} \\
\end{bmatrix}
=H
\begin{bmatrix}
   \tilde{x} \\
   \tilde{s} \\
\end{bmatrix}
+
\begin{bmatrix}
   0 \\
   \tilde{\eta}_\Gamma \\
\end{bmatrix}, \label{hs}
\end{align}
where $\tilde x(0)=x_0$, $\tilde{\eta}_\Gamma=e^{-\frac{\rho}{2}t}\eta _\Gamma$, and
\begin{align}
\label{H_sys}
H=
\begin{bmatrix}
   \mathcal A & -BR^{-1}B^T \\
   Q_{\Gamma} & -\mathcal A^T\\
\end{bmatrix},\    \mathcal A=A-BR^{-1}B^T\Pi-\frac{\rho}{2}I.
\end{align}
Note that $\tilde{\eta}_\Gamma$ in (\ref{hs}) is a function of $t$.  Since
$
    Q_\Gamma
$
is symmetric, $H$ is a Hamiltonian matrix.

%\subsection{Eigenvalues and Eigenvectors of Hamiltonian Matrices}
%\label{sec:eehm}

%\begin{proposition}\cite{L84}
%\label{Ham_sym}
%Let $K$ be a Hamiltonian matrix and $p_K(x)$ its  characteristic polynomial.
%If $z_0\in\mathbb C$ is a root of $p_K(x)=0$, then
%$
%p_K(-z_0)=p_K(\overline z_0)=p_K(-\overline z_0)=0,
%$ where $\overline z_0$ is the conjugate of $z_0$.
%\end{proposition}

\subsection{Existence and Uniqueness of a Solution}
\label{sec:eubs}

Consider the general matrix differential equation
\begin{align}\label{odez}
\frac{dz}{dt}= Kz+\psi(t),
\end{align}
where $z=[z_1^T, z_2^T]^T$, $z_1,z_2\in \mathbb{R}^n$,
$K\in \mathbb{R}^{2n\times 2n}$ and for some $C>0$, $|\psi(t)|\le C e^{-\frac{\rho t}{2}}$ for all $t\ge 0$, and where $z_1(0)$ is given.

\begin{definition}
The matrix $K\in \mathbb{R}^{2n\times 2n}$ is said to satisfy condition (H0) if  there exists an invertible real matrix $U=(U_{ij})_{1\le i,j\le 2}$,
 where $U_{11}\in \mathbb{R}^{n\times n}$ is invertible, such that
$$
U^{-1}K U=
\begin{bmatrix}
   F_{11} & F_{12} \\
   0 & F_{22}\\
\end{bmatrix},
$$
where $F_{11}$ and $-F_{22}$ are $n\times n$ stable matrices.
\end{definition}

Let $M_1$ be an  $(n+m)\times (n+m)$ real (or complex) matrix which has an $n$-dimensional invariant subspace ${\mathcal V}$.  If ${\mathcal V}$ is spanned by the columns of an $(n+m)\times n$ matrix whose  leading $n\times n$ sub-matrix is invertible, ${\mathcal V}$ is called a graph subspace \cite{BIM11,LR95}.

\begin{remark}
A matrix $K$ satisfying (H0) has $n$ stable
eigenvalues and the associated $n$-dimensional stable invariant
subspace of $K$  is a graph subspace.
\end{remark}

 \begin{lemma}\label{EUBS}
 Suppose $K$ in \eqref{odez} satisfies (H0). Then for the given $z_1(0)$, there exists a unique
\begin{align}
z_2(0)=&
U_{21}U_{11}^{-1} z_1(0) \nonumber\\
&+(U_{21}U_{11}^{-1} U_{12}-U_{22})
\int_0^\infty e^{-F_{22} \tau}[V_{21},V_{22}]\psi(\tau) d\tau \nonumber
 \end{align}
such that \eqref{odez} has a bounded solution on $[0,\infty)$,
where $V=U^{-1}= (V_{ij})_{1\le i,j\le 2}$. In this case,   for some $\epsilon_0>0$,
$ z(t) e^{\varepsilon_0 t}$ is still bounded on $[0,\infty)$.
 \end{lemma}

\begin{proof} For \eqref{odez}, we apply a change of variable to define
\begin{align}
y=U^{-1} z, \label{yz}
\end{align}
where $y=[y_1^T, y_2^T]^T$, $y_i\in \mathbb{R}^n$.
We have
\begin{align}
&\frac{dy_1}{dt}= F_{11} y_1+F_{12} y_2 +\varphi_1(t),\\
&\frac{dy_2}{dt}= F_{22} y_2 +\varphi_2(t),
\end{align}
where $U^{-1} \psi =[\varphi_1^T, \varphi_2^T]^T$.
We proceed to find a bounded solution $y$. Since $-F_{22}$ is stable, there is a unique choice of
$$
y_2(0)= -\int_0^\infty e^{-F_{22}\tau } \varphi_2(\tau) d\tau
$$
such that $y_2(t)=-\int_t^\infty e^{F_{22}(t-\tau)} \varphi_2(\tau) d\tau $ is bounded, which further determines a bounded $y_1$ regardless of  $y_1(0)$. Using the relation \eqref{yz} at $t=0$, we next uniquely determine \begin{align}\label{y10}
y_1(0)= U_{11}^{-1} [z_1(0)- U_{12} y_2(0)].
\end{align}
 Finally, we obtain
$z_2(0)=U_{21}y_1(0)+U_{22}y_2(0)$, which gives  a bounded solution $z$ of \eqref{odez}
on $[0,\infty)$. It can be checked that for some $\epsilon_0>0$,
$ z(t) e^{\varepsilon_0 t}$ is still bounded on $[0,\infty)$.

The choice of $z_2(0)$ is unique since otherwise by \eqref{yz} we may construct two different bounded solutions $y\ne  \hat y$, where we necessarily have $y_1(0)\ne \hat y_1(0)$, $y_2(0)=\hat y_2(0)$, which is impossible in view of \eqref{y10}.  \end{proof}

%\begin{lemma}
%\label{EUBS}
%Assume there exists an invertible real matrix $U=(U_{ij})_{1\le i,j\le 2}$,
% where $U_{11}\in \mathbb{R}^{n\times n}$ is invertible, such that
%$$
%U^{-1}HU=
%\begin{bmatrix}
%   \mathcal{A}_C & F \\
%   0 & -\mathcal{A}^T_C\\
%\end{bmatrix},
%$$
%where $\mathcal{A}_C$ is stable. Then the following assertions hold:

%i) For any fixed  $x_0$ of  (\ref{SCE1_1})-(\ref{SCE2_1}), there exists %  a unique $s_0$ such that  (\ref{SCE1_1})-(\ref{SCE2_1}) has a solution %$(\bar x,s)\in C_{\rho/2}([0,\infty),\mathbb R^{2n})$,
%
%ii)  $s_0$ is unique and is equal to
%
% $$s_0=U_{21}U_{11}^{-1} x_0+(U_{21}U_{11}^{-1}U_{12}-U_{22})\int_{0}^{+\infty}e^{A_C^T\tau}V_{22}\eta_\Gamma %e^{-\rho\tau/2}d\tau,$$ where $V=U^{-1}= (V_{ij})_{1\le i,j\le 2}$.

%2) The above initial value $s_0=U_{21}U_{11}^{-1}(\bar x_0-U_{12}\hat s_0)+U_{11}-\int_{0}^{+\infty}e^{\mathcal{A}^T_C\tau}
%U_{11}^T\tilde{\eta}_\Gamma d\tau$.

%\end{lemma}

%\begin{proof}

 The proof of the existence result in the theorem below  reduces to showing the stable invariant subspace of $H$ is a graph subspace.

\begin{theorem}
\label{E_np}
Assume that the pair $(A,B)$ is stabilizable and the Hamiltonian matrix $H$  in \eqref{H_sys}  has no  eigenvalues with zero real parts. Then there exists a unique initial condition $s_0$ such that \eqref{SCE1_1}-\eqref{SCE2_1}
 %(\ref{SCE1})-(\ref{SCE2})
 has a  solution $(\bar x,s)\in C_{\rho/2}([0,\infty),\mathbb R^{2n})$.
\end{theorem}

\begin{proof}
Since $\mathcal A$ is stable, both $({\mathcal A}, B)$ and $(\mathcal A,BR^{-1}B^T)$ are stabilizable \cite{W12}.
Consider the ARE
\begin{align}
X\mathcal A+\mathcal A^TX-XBR^{-1}B^TX-Q_\Gamma=0. \label{CARE_pr2}
\end{align}

By Corollary \ref{cor:ric}, there exists a unique maximal real symmetric solution $X_+$ such that $\mathcal A-BR^{-1}B^TX_+$ is stable.
Denote $U=
\begin{bmatrix}
I & 0\\
X_+ & I
\end{bmatrix}
$.
So
$U^{-1}=
\begin{bmatrix}
I & 0\\
-X_+ & I
\end{bmatrix}.
$ Then
\begin{align*}
U^{-1}HU&=\begin{bmatrix}
I & 0\\
-X_+ & I
\end{bmatrix}
\begin{bmatrix}
\mathcal A & -BR^{-1}B^T \\
Q_{\Gamma} & -\mathcal A^T\\
\end{bmatrix}
\begin{bmatrix}
I & 0\\
X_+ & I
\end{bmatrix}\\
%&=
%\begin{bmatrix}
%\mathcal A-BR^{-1}B^TX_+ & -BR^{-1}B^T\\
%X_+BR^{-1}B^TX_+-X_+A-A^TX_++Q_\Gamma & -(\mathcal A-BR^{-1}B^TX_+)^T
%\end{bmatrix}\\
&=
\begin{bmatrix}
\mathcal A_C & -BR^{-1}B^T\\
0 & -\mathcal A_C^T
\end{bmatrix},
\end{align*}
where $\mathcal A_C=\mathcal A-BR^{-1}B^TX_+$ is stable.
By Lemma \ref{EUBS}, after selecting the initial condition $s_0=X_+ x_0-\int_{0}^{\infty}e^{\mathcal{A}^T_C\tau}\eta_\Gamma e^{-\rho\tau/2} d\tau$, the resulting  solution $(\bar x,s)\in C_{\rho/2}([0,\infty),\mathbb R^{2n})$. And $s_0$ is unique.
\end{proof}

For the special case $Q_\Gamma\le 0$,  since ${\mathcal A}$ is stable, \eqref{CARE_pr2} has a unique positive semi-definite solution $X$ by the standard theory of Riccati equations \cite{W12}. On the other hand, by \cite[Theorem 9.3.3]{LR95}, in this case $H$ necessarily has no eigenvalues with zero real parts.

%Notice that in the proof of Theorem \ref{E_np}, $\mathcal A=A-BR^{-1}B^T\Pi
%-\frac{\rho}{2}I$ %corresponds to $A$ in Theorem \ref{MSS}; $-Q_\Gamma$ %corresponds to $Q_o$ %and $BR^{-1}B^T$ corresponds to $M$.

\begin{example}
\label{ex_h2x2}
Consider a scalar model with $A=a$, $B=b\ne 0$, $R=r>0$, $Q=q> 0$, $\Gamma=\gamma$. Then $Q_\Gamma=(2\gamma-\gamma^2)q$. Denote $a_\rho = a-\rho/2$ and $b_r= b/\sqrt{r}$. We solve the Riccati equation
%\begin{align}
$\rho\Pi=2a\Pi-b^2\Pi/r+q$ %\label{eqnPI}
%\end{align}
 to obtain $\Pi= (a_\rho +\sqrt{a_\rho^2 +qb_r^2})/b_r^2$.
Then $H$ in \eqref{H_sys} becomes
 \begin{align*}
H=\begin{bmatrix}
-\sqrt{a_\rho^2+qb_r^2} & -b_r^2\\
(2\gamma- \gamma^2)q & \sqrt{a_\rho^2+qb_r^2}
\end{bmatrix}.
\end{align*}
The characteristic equation $\det (\lambda I- H) =0 $ reduces to   $\lambda^2=a_\rho^2+qb_r^2(1-\gamma)^2$.
Therefore, $H$ has eigenvalues with zero real parts if and only if
 $a=\rho/2$ and $\gamma=1$ when $b\ne 0$ and $q>0$.
\end{example}

\begin{example}
\label{ex_img}
 We continue with the system in  Example  \ref{ex_h2x2} for the case  $a=\rho/2$ and $\gamma=1$.
The SCE equation system now becomes
\begin{align*} %\label{Hodemulti}
\begin{bmatrix}
\dot{\bar x}(t)\\
\dot s(t)
\end{bmatrix} =
\begin{bmatrix}
\frac{\rho}{2}-\sqrt{q}|b_r| &  -b_r^2 \\
q &  \frac{\rho}{2} + \sqrt{q} |b_r|
\end{bmatrix}
\begin{bmatrix}
{\bar x}(t)\\
 s(t)
\end{bmatrix},
\end{align*}
where $\bar x(0)$ is given. We obtain the solution
\begin{align*}
\begin{bmatrix}
{\bar x}(t)\\
 s(t)
\end{bmatrix} = e^{ \frac{\rho t}{2}} \begin{bmatrix}
1-\sqrt{q}|b_r|t &  -b_r^2 t \\
qt & 1+ \sqrt{q} |b_r| t
\end{bmatrix} \begin{bmatrix}
{\bar x}(0)\\
 s(0)
\end{bmatrix},
\end{align*}
which is not in $C_{\rho/2}([0,\infty),\mathbb R^{2})$  unless  $\bar x(0)= s(0)=0$.\end{example}

\begin{example}
\label{ex_margstab}
Consider the system given in Example \ref{ex_img}. We have $\Pi=\sqrt{q}/|b_r|$, ${\mathcal A}=-\sqrt{q} |b_r|$,
$Q_\Gamma= q$. The Riccati equation \eqref{CARE_pr2} now has the solution $X=-\sqrt{q}/|b_r|<0$, and ${\mathcal A}-b_r^2X=0$ implying $X$ being almost stabilizing, which is due to the two zero eigenvalues of $H$.
\end{example}

% Thus the SCE equation system \eqref{Hodemulti} has an infinite number of %solutions.
%For optimizing the social cost via the resulting control law \eqref{oc}, %one needs to further refine the choice of $s(0)$ by treating it as an optimization %parameter.
%This non-uniqueness result may be interpreted as  a result of lack of detectability %for the full system of $N$ players.

\subsection{Computational Methods for the ARE}
\label{sec:cm}

Consider ARE (\ref{CARE}).
% which is rewritten below:
%\begin{align}
%X A_o+{A}_o^TX-XMX+Q_o=0.\label{CARE_HAM}%
%\end{align}
 Let  $H_o$ be defined by \eqref{Ho}. This part describes the numerical  method
for a stabilizing solution when $Q_o$ may not be positive semi-definite.
Denote
$$W=\begin{bmatrix}U_1\\U_2\end{bmatrix}\in \mathbb C^{2n\times n}.
$$

\begin{proposition}
\label{cx_ee}
Suppose i) $H_o$ has no eigenvalues with zero real parts and
\begin{align}
H_o\begin{bmatrix} U_1\\ U_2\end{bmatrix}=
\begin{bmatrix} U_1\\U_2 \end{bmatrix}S_o, \label{HoS2}
\end{align}
where $S_o$ is stable; ii) $(A_o, M)$ is stabilizable.  Then $U_1$ is invertible and $U_2U_1^{-1}$ is real, symmetric and satisfies (\ref{CARE}), and $A_o-MU_2U_1^{-1}$ is stable.
\end{proposition}
\begin{proof} This proposition holds as a corollary to
Theorems 13.5 and 13.6 in \cite{ZDG96} under condition ii).
In this case the invariant subspace of $H_o$ associated with the $n$ stable eigenvalues is a (complex) graph subspace, and $U_1$ is necessarily invertible.  \end{proof}

   %  \begin{proof}
In fact,  by  Proposition  \ref{cor:ric}, there exists $X$ satisfying (\ref{CARE}) such that $A_o-MX$ is stable. %So By Lemma \ref{c_l_2},
 It is straightforward to verify \cite{BIM11}
 $$H_o\begin{bmatrix} I \\
 X\end{bmatrix}=\begin{bmatrix} I \\ X\end{bmatrix}(A_o-MX).$$

\begin{remark}
 Since $H_o$ has $n$ eigenvalues in the open left and right half planes, respectively, there exist
  $U_1, U_2$
to satisfy condition i) in Proposition \ref{cx_ee}.
\end{remark}

%Theorem \ref{cx_ee} guarantees that when we find a set of generalized eigenvectors% %corresponding to the stable Jordan canonical block of $H_o$, the stabilizing %solution can be found by this set of generalized eigenvectors.

 A similar method of using invariant subspace to solve a discrete-time algebraic Riccati equation was presented in \cite{PLS80}, where the state weight matrix $Q$ is positive semi-definite.

% Pappas, Laub and Sandell \cite{PLS80} used generalized eigenvectors to% %find the positive definite solution of the  Riccati equation
%where existence is guaranteed by  stabilizability and detectability conditions.

To  apply Proposition \ref{cx_ee} to numerically solve the ARE, one needs to first find a set of basis vectors of the stable invariant subspace of $H_o$.
 Now we introduce a  convenient method to find such a set of vectors.
\begin{proposition}\cite{L79}
Assume  the Hamiltonian matrix $H_o \in \mathbb R^{2n\times 2n} $\ has no eigenvalues with zero real parts.  Then
there exists an orthogonal transformation $W\in \mathbb R^{2n\times 2n}$ such that
$$
W^TH_oW=\begin{bmatrix}
H_{11} & H_{12}\\
0 & H_{22}
\end{bmatrix} =\widehat  H_o,
$$
where  $H_{11}\in \mathbb{R}^{n\times n} $ is a stable matrix. \endproof
\end{proposition}

We refer to  $\widehat H_o$  as the real Schur form and $W$ consists of $2n$ independent vectors which are called Schur vectors. If we partition $W$ into four $n \times n$ blocks
$\begin{bmatrix}
W_{11} & W_{12}\\
W_{21} & W_{22}
\end{bmatrix},$ $\begin{bmatrix}
W_{11}\\W_{21}
\end{bmatrix}
$ consists of $n$ Schur vectors corresponding to stable Schur block $H_{11}$
and provides  a specific choice of the vectors to span the stable invariant subspace in Proposition \ref{cx_ee} and $W_{11}^{-1}$ exists.

\section{Extension to Mean Field Games }
\label{chp:mggp}

%\subsection{Description of Problem}
%\label{sec:mfgd}

We consider a Nash game  of $N$ players with dynamics and costs given by \eqref{doa}-\eqref{ic}.
By mean field game theory \cite{HCM03, HCM07,HCM12}, the decentralized strategies for the game may be designed by using the following ODE system:
\begin{numcases}{}
    \frac{d\bar{x}}{dt}=(A-BR^{-1}B^T\Pi)\bar{x}-BR^{-1}B^Ts,\label{MFG_H1}\\
    \frac{ds}{dt}=Q\Gamma\bar{x}+(\rho I-A^T+\Pi BR^{-1}B^T)s+Q\eta,\label{MFG_H2}
\end{numcases}
where $\bar x(0)=x_0$ is given.
Define $\tilde x=e^{-\rho t/2}\bar{x}$ and $\tilde s=e^{-\rho t/2}s$. We obtain
\begin{numcases}{}
    \frac{d\tilde x}{dt}=\mathcal{A}\tilde x-BR^{-1}B^T\tilde s,\label{MFG_H3}\\
    \frac{d\tilde s}{dt}=Q\Gamma\tilde x-\mathcal{A}^Ts+\tilde\eta,\label{MFG_H4}
\end{numcases}
where $\tilde x(0)=x_0$,  $\mathcal{A}=A-BR^{-1}B^T\Pi-\frac{\rho}{2}I$, $\tilde\eta=e^{-\rho t/2}Q\eta$.

Notice that $Q\Gamma$ is generally asymmetric and the
coefficient matrix in \eqref{MFG_H3}-\eqref{MFG_H4} does not have a Hamiltonian structure.
 However, we can apply the invariant subspace method in Section \ref{sec:eubs} to find a solution $(\bar x, s)\in C_{\rho/2}([0,\infty),\mathbb R^{2n})$.
Denote
\begin{align}
M_{\rm mfg}=\begin{bmatrix}
\mathcal A & -BR^{-1}B^T\\
Q\Gamma & -\mathcal A^T
\end{bmatrix}. \label{mfgM}
\end{align}

\begin{theorem}
\label{EUBS_M} Suppose  $M_{\rm mfg}$ in \eqref{mfgM}
satisfies condition (H0) with $U=(U_{ij})_{1\le i,j\le 2}$, where $U_{11}$ is invertible,
such that
$$ U^{-1}M_{\rm mfg}U=
\begin{bmatrix}
   M_{11} & M_{12} \\
   0 & M_{22}\\
\end{bmatrix},
$$
where $M_{11}$  and $-M_{22}$ are stable.
 Then for any given $x_0$ in (\ref{MFG_H1})-(\ref{MFG_H2}), there exists a unique
$$\hskip -0.5 eM s_0  =U_{21}U_{11}^{-1} x_0+(U_{21}U_{11}^{-1}U_{12}-U_{22})
 \int_{0}^{\infty}e^{-M_{22}\tau}V_{22}Q\eta e^{-\frac{\rho}{2}\tau}d\tau,   $$
 where $V=U^{-1}=(V_{ij})_{1\le i,j\le 2}$,
 such that  (\ref{MFG_H1})-(\ref{MFG_H2}) has a  solution $(\bar x,s)\in C_{\rho/2}([0,\infty),\mathbb R^{2n})$.
%2) The above initial value $s_0=U_{21}U_{11}^{-1}
%(\bar x_0-U_{12}\hat s_0)+U_{11}-\int_{0}^{+\infty}
%e^{\mathcal{A}^T_C\tau}U_{11}^T\tilde{\eta}_\Gamma d\tau$.

\end{theorem}

\begin{proof} The theorem follows from Lemma \ref{EUBS}.
\end{proof}

\section{Numerical Examples}
\label{sec:num}

\subsection{Riccati Equation and SCE Equation System }

%\subsubsection{Computational Results}
\label{subs:CR}

Consider ARE \eqref{CARE_pr2},
%\begin{align}
%\label{ex_sx}
%X\mathcal A+\mathcal A^TX-XBR^{-1}B^TX-Q_\Gamma=0,
%\end{align}
where $\mathcal A=A-BR^{-1}B^T\Pi-\frac{\rho}{2}I$. We compute the stabilizing solutions of \eqref{ARE} and \eqref{CARE_pr2} and further solve the SCE equation system.
In the examples, we specify the system parameters, including the matrix $A$, which further determine $\mathcal A$. The computation follows the notation in Theorem  \ref{E_np} and its proof.
\begin{example}
\label{ex_meth1}
Consider the scalar system: $A=2,\ B=1,\ Q=2,\ R=1,\ \eta=1,\ \rho=1,\ \Gamma=1$ and the initial condition $x_0=1$. We have $Q_\Gamma=Q>0$ and $\Pi=3.5616$.
The SCE equation system \eqref{SCE1_1}-\eqref{SCE2_1}
 %\eqref{SCE1}-\eqref{SCE2}
 becomes
$$
\begin{bmatrix}
   \frac{d\bar{x}}{dt} \\
   \frac{ds}{dt} \\
\end{bmatrix}
=
\begin{bmatrix}
   -1.5616 & -1.0000\\
    2  & 2.5616
\end{bmatrix}
\begin{bmatrix}
   \bar{x} \\
   s \\
\end{bmatrix}
,
$$ and
$$H=
\begin{bmatrix}
   -2.0616 & -1.0000\\
    2.0000 &  2.0616
\end{bmatrix}.
$$
The eigenvalues of $H$ are $-1.5$ and $1.5$, which have no zero real parts.
By solving  \eqref{CARE_pr2} using Schur vectors, we obtain $X_+=-0.5615$ and  $\mathcal A_C=-1.5$.

We select  $s_0=X_+ x_0-\int_{0}^{\infty}e^{(\mathcal{A}^T_C-\frac{\rho}{2}I)\tau}\eta_\Gamma d\tau=-0.5615$. Under the initial condition $(x_0,s_0)=(1, -0.5615)$, we obtain  $(\bar x(t), s(t))=(e^{-t},-0.5616e^{-t})\in C_{1/2}([0,\infty),\mathbb R^{2})$.

%\begin{figure}[t]
%\begin{center}
%\begin{tabular}{c}
%\psfig{file=Figure1.eps, width=6in, height=3.6in}
%\end{tabular}
%\end{center}
%\caption{ Solution when $\gamma=1$} \label{ex1_1}
%\end{figure}

\end{example}

\begin{example}
\label{ex_222}
Consider the system with parameters
\begin{align*}
&A=
\begin{bmatrix}
1 & -1\\
0 & 2
\end{bmatrix},\
B=
\begin{bmatrix}
1\\1
\end{bmatrix},\
Q=
\begin{bmatrix}
1&0\\
0&-0.5
\end{bmatrix},\
\eta=
\begin{bmatrix}
1\\0
\end{bmatrix},\\
&
\Gamma=\gamma\begin{bmatrix}
1 & 0\\
0.5 & 1
\end{bmatrix},\ \rho=1,\  R=1,\  \gamma=2
\end{align*}
 and initial condition $x_0=\begin{bmatrix}
1\\1
\end{bmatrix}.
$
We have $Q_\Gamma=
\begin{bmatrix}
0.5 & 0.5\\
0.5 & 0
\end{bmatrix}.
$  Both $Q$ and $Q_\Gamma$ are indefinite.
We solve
$$
\Pi= \begin{bmatrix}
 3.5483 &  -5.6810 \\
  -5.6810 &  12.6724
\end{bmatrix}.
$$
% which is indefinite with
% eigenvalues  $-0.3090$ and $0.8090$.
The SCE equation system is
\begin{align*}
\begin{bmatrix}
   \frac{d\bar{x}_1}{dt} \\
   \frac{d\bar{x}_2}{dt} \\
   \frac{ds_1}{dt}\\
   \frac{ds_2}{dt} \\
\end{bmatrix}
=
(H+\frac{\rho}{2} I)
& \begin{bmatrix}
   \bar{x}_1 \\
   \bar{x}_2 \\
   s_1 \\
   s_2 \\
\end{bmatrix}
+
\begin{bmatrix}
0\\0\\-1\\0
\end{bmatrix}
,
\end{align*}
 and the Hamiltonian matrix
\begin{align*}
H&=
\begin{bmatrix}
    2.6327  & -7.9914  & -1.0000  & -1.0000\\
    2.1327  & -5.4914  & -1.0000  & -1.0000\\
    0.5000  &  0.5000  & -2.6327  & -2.1327\\
    0.5000  &       0  &  7.9914  &  5.4914
\end{bmatrix}.
\end{align*}
The eigenvalues of $H$ are $-1.0655 \pm 0.6208i$, $1.0655 \pm 0.6208i$, so $H$ has no eigenvalues with zero real parts.
Solving  \eqref{CARE_pr2} with Schur vectors, we have
$$X_+=\begin{bmatrix}
   -2.0373  &  2.7519\\
    2.7519  & -4.1941
\end{bmatrix},\quad
\mathcal A_C=\begin{bmatrix}
    1.9181  & -6.5492\\
    1.4181  & -4.0492
\end{bmatrix}.$$

We select $$s_0=X_+
x_0-\int_{0}^{\infty}e^{(\mathcal{A}^T_C-\frac{\rho}{2}I)\tau}\eta_\Gamma d\tau=\begin{bmatrix}
  2.3185\\
 -3.7513
\end{bmatrix}.$$
 For the initial condition $(1,1$, $2.3185,-3.7513)$, we compute the  solution $(\bar x, s)$ which is displayed in Fig. \ref{ex2_1}.
\begin{figure}[t]
\begin{center}
\begin{tabular}{c}
\psfig{file=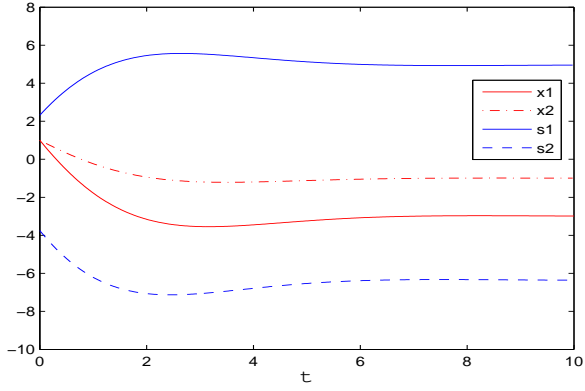, width=3.6in, height=2.2in}\\
\end{tabular}
\end{center}
\caption{ Solution of $(\bar x, s)$ in Example \ref{ex_222} } \label{ex2_1}
\end{figure}

\end{example}

\subsubsection{Comparison}
\label{subs:C}
We compare with a fixed point method, which is used to analyze the SCE equation system by verifying a contraction condition in \cite{HCM12}.
Consider  (\ref{hs}).
By the method in \cite{HCM12,NCMH12}, the solution $\tilde x(t)$ is a fixed point to the equation
$x(\cdot)=\Psi( x(\cdot))$
where
\begin{align*}
&[\Psi(x)](t)=
e^{\mathcal At}x_0+\int_0^t e^{\mathcal A(t-s)}BR^{-1}B^T\\
&\qquad\qquad\qquad\cdot\left[ \int_s^\infty e^{-\mathcal A^T(s-\tau)}\left(Q_\Gamma x(\tau)
+e^{-\frac{\rho}{2}\tau}\eta_\Gamma\right)d\tau\right]ds,
\end{align*}
where we look for $x(\cdot)\in C_b([0,\infty), \mathbb{R}^n)$, i.e., the set of bounded and continuous functions on $[0, \infty)$ with norm $|x|_\infty= \sup_{t\ge 0}|x(t)|$.
The fixed point exists and is unique if there exists  $\beta\in(0,1)$ such that
$
|\Psi (x)-\Psi (y)|_\infty\leq\beta|x-y|_\infty.
$
 Let $\|\cdot\|$ denote the Frobenius norm.
We have the estimate
\begin{align*}
&\|[\Psi (x)](t)-[\Psi (y)](t)\|\\
&=
\|\int_0^t e^{\mathcal A(t-s)}BR^{-1}B^T\\
&\qquad\qquad\cdot\left\{ \int_s^\infty e^{-\mathcal A^T(s-\tau)}
\left[Q_\Gamma\left(x(\tau)-y(\tau)\right)\right]d\tau\right\}ds\|\\
&\leq\left\|x-y\right\|\int_0^\infty \|e^{\mathcal As}BR^{-1}B^T\|\left( \int_0^\infty \|e^{\mathcal A^T\tau}Q_\Gamma\|d\tau\right)ds.
\end{align*}
Let $\beta=\int_0^\infty \left\|e^{\mathcal As}BR^{-1}B^T\right\|( \int_0^\infty \|e^{\mathcal A^T\tau}Q_\Gamma\|d\tau)ds$.  We note that the upper bound estimate may not be tight.

For Example \ref{ex_222} with $\gamma=2$, we numerically obtain $\beta=6.34694>1$, which does not validate the contraction condition. If we set $\gamma=0.05$ instead, then $\beta=0.736681<1$ implying  the contraction condition.

\subsection{The Mean Field Game}

 The next example  uses the Schur decomposition for a general square real matrix.
\begin{example}
\label{ex_mfg}
Consider $A=
\begin{bmatrix}
5 & -5\\
0 & 10
\end{bmatrix},\
B=
\begin{bmatrix}
1\\1
\end{bmatrix},\
Q=
\begin{bmatrix}
1&0\\
0&1
\end{bmatrix},\
\eta=
\begin{bmatrix}
1\\0
\end{bmatrix}, \
$
$
\Gamma=\begin{bmatrix}
5 & 0\\
2.5 & 5
\end{bmatrix}$, $\rho=2$, $R=1
$, and the initial condition is $x_0=\begin{bmatrix}
1\\1
\end{bmatrix}.
$
By \eqref{mfgM},
we calculate $$M_{\rm mfg}=
\begin{bmatrix}
   14.7999 &-42.0915 & -1.0000 & -1.0000\\
   10.7999 &-28.0915 & -1.0000 & -1.0000\\
    5.0000 &       0 &-14.7999 &-10.7999\\
    2.5000 &  5.0000 & 42.0915 & 28.0915
\end{bmatrix},
$$
which has eigenvalues $9.2522$, $1.7783$, $-2.0950$ and $-8.9356$. The Schur decomposition represents $M_{\rm mfg}$ as
\begin{align*}
U\begin{bmatrix}
   -8.9356 &-13.4806 & 27.2557 & -35.7593\\
         0 & -2.0950 &-46.5898 &-32.5048\\
         0 &       0 &  9.2522 &-15.8{}037\\
         0 &       0 &       0 &  1.7783
\end{bmatrix}U^{-1},
\end{align*}
where
\begin{align*}U=
\begin{bmatrix}
   -0.5425 & -0.6081 & -0.0928 &  0.5721\\
   -0.3060 & -0.4284 & -0.3027 & -0.7945\\
    0.5543 & -0.1947 & -0.7863 &  0.1911\\
   -0.5521 &  0.6394 & -0.5305 &  0.0700
\end{bmatrix}.
\end{align*}
For $U_{11}=\begin{bmatrix}
   -0.5425 & -0.6081\\
   -0.3060 & -0.4284
\end{bmatrix}
$,
 $\det(U_{11})=0.0464>0$, so $U_{11}$ is invertible. We select the initial condition
\begin{align*}
s_0&=U_{21}U_{11}^{-1} x_0+(U_{21}U_{11}^{-1}U_{12}-U_{22})\int_{0}^{\infty}e^{-M_{22}\tau}V_{22}Q\eta e^{-\frac{\rho\tau}{2}} d\tau\\
&=\begin{bmatrix}
2.31075 \\
-4.11538
\end{bmatrix}.
\end{align*}
Fig. \ref{ex3_1} shows the solution  $(\bar x,s)$ for \eqref{MFG_H1}-\eqref{MFG_H2}.

\begin{figure}[t]
\begin{center}
\begin{tabular}{c}
\psfig{file=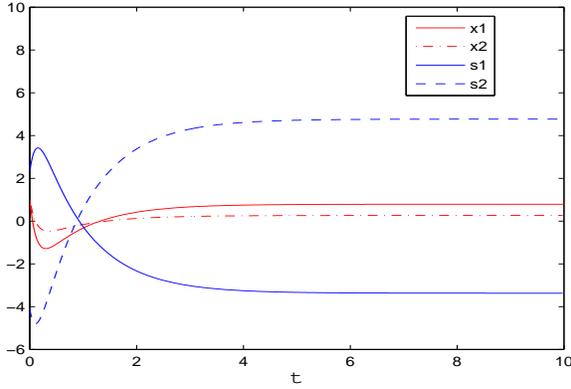, width=3.6in, height=2.2in}\\
\end{tabular}
\end{center}
\caption{ Solution of $(\bar x, s)$ in Example \ref{ex_mfg}} \label{ex3_1}
\end{figure}
\end{example}

\section{Conclusion}
\label{chp:con}
This paper develops a methodology to prove the existence and uniqueness of the solution of the LQ social optimization problem when the corresponding Hamiltonian matrix has no eigenvalues on the imaginary axis. We also develop a numerical method  for solving the ODE system by  applying an invariant subspace method.
 We further extend the invariant subspace method to solve LQ mean field games.

\end{document}